\documentclass[12pt]{amsart}
\usepackage{amscd,amsmath,amsthm,amssymb}
\usepackage[left]{lineno}
\usepackage{color}
\usepackage{stmaryrd}
\usepackage[utf8]{inputenc}
\usepackage{cleveref}
\usepackage{epstopdf}
\usepackage{graphicx}
\usepackage{xcolor}

\usepackage{comment}

\usepackage{tikz}
\usetikzlibrary{cd}

\definecolor{verylight}{gray}{0.97}
\definecolor{light}{gray}{0.9}
\definecolor{medium}{gray}{0.85}
\definecolor{dark}{gray}{0.6}

 \def\G{{\mathcal G}}
 \def\F{{\mathcal F}}
 
 \def\D{{\mathcal D}}
 \def\B{{\mathcal B}}

  \def\I{{\mathcal I}}

 \def\0b{{\mathbf 0}}

\def\Nb{{\mathbb N}}

\def\reg{{\mathbf reg}}

 \def\opn#1#2{\def#1{\operatorname{#2}}} 
 \opn\chara{char} \opn\length{\ell} \opn\pd{pd} \opn\rk{rk}
 \opn\projdim{proj\,dim} \opn\injdim{inj\,dim} \opn\rank{rank}
 \opn\depth{depth} \opn\grade{grade} \opn\height{height}
 \opn\embdim{emb\,dim} \opn\codim{codim}
 
 \opn\Tr{Tr} \opn\bigrank{big\,rank}
 \opn\superheight{superheight}\opn\lcm{lcm}
 \opn\trdeg{tr\,deg}
 \opn\reg{reg} \opn\lreg{lreg} \opn\ini{in} \opn\lpd{lpd}
 \opn\size{size} \opn\sdepth{sdepth}
 \opn\link{link}\opn\fdepth{fdepth}\opn\lex{lex}
 \opn\tr{tr}
 \opn\type{type}
 \opn\gap{gap}
 \opn\arithdeg{arith-deg}
 \opn\HS{HS}
 \opn\GL{GL}
 \opn\div{div} \opn\Div{Div} \opn\cl{cl} \opn\Cl{Cl}
 \opn\Spec{Spec} \opn\Supp{Supp} \opn\supp{supp} \opn\Sing{Sing}
 \opn\Ass{Ass} \opn\Min{Min}\opn\Mon{Mon}
 \opn\Ann{Ann} \opn\Rad{Rad} \opn\Soc{Soc}\opn\Deg{Deg}
 \opn\Im{Im} \opn\Ker{Ker} \opn\Coker{Coker} \opn\Am{Am}
 \opn\Hom{Hom} \opn\Tor{Tor} \opn\Ext{Ext} \opn\End{End}
 \opn\Aut{Aut} \opn\id{id}
 
 \opn\nat{nat}
 \opn\pff{pf}
 \opn\Pf{Pf} \opn\GL{GL} \opn\SL{SL} \opn\mod{mod} \opn\ord{ord}
 \opn\Gin{Gin} \opn\Hilb{Hilb}\opn\sort{sort}
 \opn\PF{PF}\opn\Ap{Ap}
 \opn\mult{mult}
 \opn\bight{bight}
 \opn\aff{aff}
 \opn\relint{relint} \opn\st{st}
 \opn\lk{lk} \opn\cn{cn} \opn\core{core} \opn\vol{vol}  \opn\inp{inp} \opn\nilpot{nilpot}
 \opn\link{link} \opn\star{star}\opn\lex{lex}\opn\set{set}
 \opn\width{wd}
 \opn\Fr{F}
 \opn\QF{QF}
 \opn\G{G}
 \opn\type{type}\opn\res{res}
 \opn\conv{conv}
 \opn\Ind{Ind}
 \opn\gr{gr}
 
 \def\pot#1#2{#1[\kern-0.28ex[#2]\kern-0.28ex]}

 \opn\dirlim{\underrightarrow{\lim}}
 \opn\inivlim{\underleftarrow{\lim}}
 \def\Implies{\ifmmode\Longrightarrow \else
         \unskip${}\Longrightarrow{}$\ignorespaces\fi}
 \def\implies{\ifmmode\Rightarrow \else
         \unskip${}\Rightarrow{}$\ignorespaces\fi}
 \def\iff{\ifmmode\Longleftrightarrow \else
         \unskip${}\Longleftrightarrow{}$\ignorespaces\fi}

 \let\:=\colon
 \newtheorem{Theorem}{Theorem}[section]
 \newtheorem{Lemma}[Theorem]{Lemma}
 \newtheorem{Corollary}[Theorem]{Corollary}
 
 \newtheorem{Remark}[Theorem]{Remark}
 
 \newtheorem{Example}[Theorem]{Example}
 
 \newtheorem{Definition}[Theorem]{Definition}
 
 \newtheorem{Conjecture}[Theorem]{Conjecture}

 \let\epsilon\varepsilon
 \let\kappa=\varkappa
 \def\qed{\ifhmode\textqed\fi
       \ifmmode\ifinner\quad\qedsymbol\else\dispqed\fi\fi}
 \def\textqed{\unskip\nobreak\penalty50
        \hskip2em\hbox{}\nobreak\hfil\qedsymbol
        \parfillskip=0pt \finalhyphendemerits=0}
 \def\dispqed{\rlap{\qquad\qedsymbol}}
 
 \opn\dis{dis}
 \def\pnt{{\raise0.5mm\hbox{\large\bf.}}}
 
 \opn\Lex{Lex}

\begin{document}

\title{The facet ideals of  chessboard complexes}
\author{Chengyao Jiang, Yakun Zhao, Hong Wang  and Guangjun Zhu$^{^*}$}

\address{Authors' address:  School of Mathematical Sciences, Soochow University, Suzhou 215006, P. R. China}

\email{1562420606@qq.com(Chengyao Jiang),1768868280@qq.com(Yakun Zhao),\linebreak[4]651634806@qq.com(Hong Wang),zhuguangjun@suda.edu.cn(Corresponding author:Guangjun Zhu).}

\thanks{$^{\ast}$ Corresponding author}


\begin{abstract}
    In this paper we describe the irreducible decomposition of the facet ideal $\F(\Delta_{m,n})$ of the chessboard complex $\Delta_{m,n}$ with $n\geq m$.
We also provide some lower bounds for depth and regularity of the facet ideal  $\F(\Delta_{m,n})$. When  $m\leq 3$,  we prove that these  lower bounds can be obtained.
\end{abstract}

\maketitle

\section{Introduction}
Graph complexes have provided an important link between combinatorics and algebra, topology, and geometry (see \cite{J,W}).
 The two most important graph complexes are the matching complex and independence complex; their algebraic and topological properties
have been widely studied by many authors (see \cite{BLVZ,FH,J,M,W,Zi, Zh2}).
The chessboard complex,  which is  the matching complex   $\mathcal{M}(K_{m,n})$ of a complete bipartite graph $K_{m,n}$, is first introduced in the thesis of Garst  \cite{G}.
He showed that $\mathcal{M}(K_{m,n})$ is Cohen--Macaulay if and only if $n\geq 2m-1$. Ziegler in \cite{Zi} strengthened this result by showing that  $\mathcal{M}(K_{m,n})$ is shellable if $n\geq 2m-1$. Hence $\mathcal{M}(K_{m,n})$ has the homotopy type of a wedge of $(m-1)$-spheres if $n\geq 2m-1$.  Freidman and    Hanlon studied Betti numbers of chessboard complex  in \cite{FH}. They showed that $b_{r-1}(\Delta_{m,n})=0$ if and only if $(m-r)(n-r)>r$, and $b_{v-1}(\Delta_{m,n})>0$ if and only if $n\geq 2m-4$ or $(m,n)=(6,6),(7,7),(8,9)$,
where $b_{i}(\Delta_{m,n})$ is the $i$-th Betti number of $\Delta_{m,n}$, it equals the rank of the homology group $H_i(\Delta_{m,n})$.
The chessboard complex may be  regarded as  a simplicial complex  formed by all admissible rook configurations on an $m\times n$ chessboard, while an admissible rook configuration on a chessboard refers to
a subset of squares of the chessboard such that no two squares lie in the same row or in the same column. It is for this reason
that the name “chessboard complex” is used.

Given a simplicial complex $\Delta$ on the vertex set $\{x_1,\ldots,x_n\}$, we consider the polynomial ring  $S=k[x_1,\ldots,x_n]$ in $n$ variables over a field $k$.
 The Stanley-Reisner ideal of $\Delta$, denoted by $\I_{\Delta}$, is the squarefree monomial ideal of $S$ given by
$$
\I_{\Delta}=(x_{i_1}\cdots x_{i_s}\mid\{x_{i_1},\ldots, x_{i_s}\} \text{\ is not a face of \ } \Delta ).
$$
The {\it facet ideal} of $\Delta$, denoted by $\F(\Delta)$, is the squarefree monomial ideal of $S$ given by
$$
\F(\Delta)=(x_{\ell_1}\cdots x_{\ell_s}\mid\{x_{\ell_1},\ldots, x_{\ell_s}\} \text{\ is  a facet of \ } \Delta ).
$$
The Stanley-Reisner ring of $\Delta$ is $k[\Delta]=S/\I_{\Delta}$.
For any positive integers $m,n$, Bj\" orner et al. in \cite{BLVZ} showed that a chessboard complex $\Delta_{m,n}$ is $(v-2)$-connected
and $\depth\,( k[\Delta_{m,n}])=v$ for any positive integers $m,n$, where  $v=\min\{m,n,\lfloor\frac{m+n+1}{3}\rfloor\}$ and $\left\lfloor\frac{m+n+1}{3}\right\rfloor$ is the largest integer $\leq \frac{m+n+1}{3}$.

H\`a and  Van Tuyl in \cite{HT} studied  edge ideals of $m$-uniform
hypergraphs, they showed that  if $H$ is a properly-connected hypergraph with the edge ideal $I(H)$ and $c$ is
the maximal number of pairwise $(m+1)$-disjoint edges of $H$, then
$\reg\,(S/I(H))\geq c(m-1)$
and the equality holds if $H$ is triangulated. Obviously, $\F(\Delta_{m,n})$ can be regarded  as the edge ideal of an $m$-uniform hypergraph, but this $m$-uniform hypergraph is not properly-connected and has no pairwise $(m+1)$-disjoint edges.

As far as we know,  little is known about algebraic properties of the facet ideal of a chessboard complex. In this article,
we focus on  properties corresponding to the irreducible decomposition (see Theorem \ref{decomposition}), depth and regularity of the facet ideal  of a chessboard complex. Our main results are as follows:
\begin{Theorem}
Let $\Delta_{m,n}$ be a chessboard complex with $n\geq m\geq 1$.
Let $V$ be its vertex set and  $\F(\Delta_{m,n})$ be its facet ideal.
 A  prime ideal $P$  of $S$ is a minimal prime of  $\F(\Delta_{m,n})$   if and only if $P$ is  generated by a set $C$,  which is obtained from $V$ by deleting all elements in $s$ rows and all elements in  $m-1-s$ columns for some $0\leq s\leq m-1$.
\end{Theorem}

\begin{Corollary} Let $\Delta_{m,n}$ be  a chessboard complex with $n\geq m\geq 1$, one has
\begin{enumerate}
 \item[(a)] $\height\,(\F(\Delta_{m,n}))=n$,
 \item[(b)] $\dim\,(S/\F(\Delta_{m,n}))=(m-1)n$,
 \item[(c)] $\bight\,(\F(\Delta_{m,n}))=\left
\{\begin{array}{l@{\  \ }l}
\lfloor\frac{(n+1)^2}{4}\rfloor&\text{if   $n<2m-1$,}\\
{(n-m+1)m}&\text{otherwise.}
\end{array}\right.$
\end{enumerate}
 \end{Corollary}

 \begin{Theorem}
Let $\Delta_{m,n}$ be a chessboard complex with $n\geq m\geq 1$, then
\begin{enumerate}
 \item[(a)]  $\reg(S/\F(\Delta_{m,n}))\geq 2(m-1)$.
 \item[(b)] $\depth\,(S/\F(\Delta_{m,n}))\geq 2(m-1)$.
 \end{enumerate}
The equalities hold if $m\leq 3$.
\end{Theorem}

 Moradi in \cite {M} showed that  if $G$ is a graph with $m$-clique ideal $K_m(G)$ and $c_m(G)$ is the minimum number
of co-chordal subgraphs  of $G$ required to cover the $m$-cliques of $G$, then
$\reg\,(S/K_m(G))\leq c_m(G)(m-1)$. One can regard  $\F(\Delta_{m,n})$  as  the $m$-clique ideal of the complement  $G$ of the line graph
of the complete bipartite graph $K_{m,n}$ where $c_m(G)\geq 2$. This shows that  the upper bound of $\reg\,(S/K_m(G))$ given by  Moradi
 can be strict.

Our paper is organized as follows.  In the preliminary  section, we collect the necessary terminology and results from the
literature. In Section $3$, we give the irreducible decomposition of the facet ideal $\F(\Delta_{m,n})$ of a chessboard complex  $\Delta_{m,n}$. In Section $4$,
we give some exact formulas for depth  and  regularity of  powers of  $\F(\Delta_{m,n})$  under the conditions that $n\geq m$ and
$m\leq 2$. In Section $5$, we  provide some lower bounds for depth and regularity of the facet ideal  $\F(\Delta_{m,n})$ with $n\geq m$.   We also prove that these  lower bounds can be attained if  $m=3$.

Throughout the paper, we  assume  that  $x_{ij}$ is the square corresponding to the $i$-th row  and  the $j$-th column in  an $m\times n$ chessboard where $n\geq m$, $1\leq i\leq m$ and $1\leq j\leq n$. By convention,  we use the matrix $(x_{ij})_{m\times n}$ for  representing  this chessboard. For example, we can represent the  chessboard on the left  by the matrix on the right

\

\begin{center}
\begin{tabular}{lr}
\begin{tabular}{|p{1cm}|p{1cm}|p{1cm}|}
\hline
$x_{11}$ & $x_{12}$ & $x_{13}$\\
& & \\
\hline
$x_{21}$ & $x_{22}$ & $x_{23}$\\
& & \\
\hline
$x_{31}$ & $x_{32}$ & $x_{33}$\\
& & \\
\hline
\end{tabular}\qquad\qquad & \qquad$\begin{pmatrix}
x_{11} & x_{12} & x_{13}\\
\\
x_{21} & x_{22} & x_{23}\\
\\
x_{31} & x_{32} & x_{33}\\
\end{pmatrix}$
\end{tabular}
\end{center}

\vspace{1cm}

\section{Preliminaries }

In this section, we gather together the needed  notations and basic facts, which will
be used throughout this paper. However, for more details, we refer the reader to \cite{BH,F,HT1,HTT,J}.

Let $\Delta$ be a simplicial complex on the vertex set $V$. We  denote the set of its facets (maximal faces under inclusion) by $Facets\,(\Delta)$.
If $Facets\,(\Delta)=\{F_1,\ldots,F_q\}$,  we write $\Delta$ as $\Delta=\langle  F_1,\ldots,F_q \rangle $.
A {\it  subcomplex} of $\Delta$ is  a simplicial complex whose
facet set is a subset of the facet set of $\Delta$.
 If $A\subseteq V$,
the {\it   induced subcomplex}  of $\Delta$  over $A$ is the subcomplex   $\Delta|_A=\{F\in \Delta \mid F\subseteq A\}$.

Let $n$ be a positive integer,  we set $[n]=\{1,2,\ldots,n\}$. Given a chessboard complex $\Delta_{m,n}$ with a vertex set $\{x_{ij}\mid i\in [m], j\in [n]\}$, we consider
 the polynomial ring   $S=k[x_{ij}\mid i\in [m], j\in [n]]$ in the entries of a generic matrix $(x_{ij})_{m\times n}$ over a field $k$.
 The  facet ideal of $\Delta_{m,n}$ is the squarefree monomial ideal of $S$ given by
$$
\F(\Delta_{m,n})=(x_{1\ell_1}\cdots x_{m\ell_m}\mid \ell_1,\ldots, \ell_m\in [n] \text{\ are pairwise different}).
$$

\begin{Example} \label{example1} The facet ideal  of  the   chessboard complex $\Delta_{3,3}$ is
$$
\F(\Delta_{3,3})=(x_{11}x_{22}x_{33},x_{11}x_{23}x_{32},x_{12}x_{21}x_{33},x_{12}x_{23}x_{31},x_{13}x_{21}x_{32},x_{13}x_{22}x_{31}).
$$
\end{Example}

\medskip
	Let  $S=k[x_1,\ldots,x_n]$ be a polynomial ring in $n$ variables over a field $k$, $I\subset S$   a nonzero homogeneous ideal and
	$$0\rightarrow \bigoplus\limits_{j}S(-j)^{\beta_{p,j}(I)}\rightarrow \bigoplus\limits_{j}S(-j)^{\beta_{p-1,j}(I)}\rightarrow \cdots\rightarrow \bigoplus\limits_{j}S(-j)^{\beta_{0,j}(I)}\rightarrow I\rightarrow 0$$
	is a {\em minimal graded free resolution} of $I$, and $S(-j)$ is an $S$-module obtained by shifting
	the degrees of $S$ by $j$. The number
	$\beta_{i,j}(I)$, the $(i,j)$-th graded Betti number of $I$, is
	an invariant of $I$ that equals the minimal number of  generators of degree $j$ in the
	$i$th syzygy module of $I$.
	Of particular interest are the following invariants which measure the size of the minimal graded
	free resolution of $I$.
	The {\it projective dimension} of $I$, denoted by $\pd\,(I)$, is defined to be
	$$\mbox{pd}\,(I):=\mbox{max}\,\{i\ |\ \beta_{i,j}(I)\neq 0\}.$$
	The {\it regularity} of $I$, denoted by $\mbox{reg}\,(I)$, is defined by
	$$\mbox{reg}\,(I):=\mbox{max}\,\{j-i\ |\ \beta_{i,j}(I)\neq 0\}.$$

\medskip
By looking at the minimal free resolution and Auslander-Buchsbaum formula (see Theorem 1.3.3 of \cite{BH}), it is easy to obtain the following lemma:
 \begin{Lemma}\label{quotient}{\em (\cite[Lemma 1.3]{HTT})}
Let $I\subset S$ be a  nonzero proper homogeneous ideal. Then
\begin{itemize}
\item[(1)] $\depth\,(S/I)=n-\pd\,(S/I)$,
\item[(2)]$\depth\,(S/I)=\depth\,(I)-1$,
\item[(3)]$\reg\,(S/I)=\reg\,(I)-1$,
\item[(4)]$\pd\,(S/I)=\pd\,(I)+1$.
\end{itemize}
\end{Lemma}
Hence we shall work with $\reg\,(I)$ (resp. $\depth\,(I)$) and $\reg\,(S/I)$ (resp. $\depth\,(S/I)$) interchangeably.

\medskip
For a monomial ideal $I$,  $\mathcal{G}(I)$ denotes the unique minimal set
of monomial generators of  $I$.

\begin{Definition} Let $I\subset S$ be a squarefree monomial ideal with irreducible decomposition
$$
I=P_1\cap P_2\cap\cdots \cap P_s,
$$
where $s\geq 1$ and  $P_i$ is a prime ideal of $S$ for any $1\leq i \leq s$.
The {\it Alexander dual} of $I$, denoted by $I^{\,\vee}$, is a squarefree monomial ideal
$$
I^{\,\vee}=({\bf x}_{P_1},{\bf x}_{P_2},\ldots,{\bf x}_{P_s}),
$$
where ${\bf x}_{P_i}=\prod\limits_{x_j\in \mathcal {G}(P_i)}\!\!x_j$ for any $1\leq i \leq s$.
\end{Definition}

\begin{Lemma}{\em (\cite[Theorem 2.18]{H})}\label{ha}
Let $I\subset S$ be a squarefree monomial ideal. Then
$$\pd\,(S/I)=\reg\,(I^{\,\vee}).$$
\end{Lemma}

In this paper we shall use the following results.

 \begin{Lemma}
\label{sum1}{\em (\cite[Lemmas 2.2 and 3.2]{HT1})}
Let $S_{1}=k[x_{1},\dots,x_{m}]$, $S_{2}=k[x_{m+1},\dots,x_{n}]$ be two polynomial rings and $S=S_1\otimes_k S_2$, let $I\subset S_{1}$,
$J\subset S_{2}$ be two nonzero homogeneous ideals. Then
\begin{itemize}
\item[(1)] $\depth\,(I+J)=\depth\,(I)+\depth\,(J)-1$,
\item[(2)]$\depth\,(JI)=\depth\,(I)+\depth\,(J)$,
\item[(3)] $\reg\,(I+J)=\reg\,(I)+\reg\,(J)-1$,
\item[(4)]$\reg\,(JI)=\reg\,(I)+\reg\,(J)$.
\end{itemize}
\end{Lemma}

\begin{Lemma}\label{sqm}{\em(\cite[Proposition 4.1]{LM})}
Let $I\subset S$ be a squarefree monomial ideal.
If each element of $\mathcal{G}(I)$ contains at least one variable not dividing any
other element of $\mathcal{G}(I)$. Then
$$reg\,(I)=|\text{supp}\,(I)|-|\mathcal {G}(I)|+1 $$
where  $\text{supp}\,(I)=\{x_i:\,  x_i|u \text{\ for  some \ } u\in \mathcal {G}(I)  \}$.
\end{Lemma}

\begin{Lemma}
\label{depthreglemma} {\em (\cite[Corollary 2.12]{MMVV})}  Let $I\subset S$ be  a  monomial ideal and $f$  a monomial of degree $d$. Then
$$\reg\,(S/(I,f))\leq\reg\,(S/I)+d-1.$$
\end{Lemma}

 \begin{Lemma}\label{sum}{\em (\cite[Theorem 1.1]{NV})}
Let $S_{1}=k[x_{1},\dots,x_{m}]$, $S_{2}=k[x_{m+1},\dots,x_{n}]$ be two polynomial rings and $S=S_1\otimes_k S_2$, let  $I\subset S_{1}$,
$J\subset S_{2}$ be two nonzero monomial ideals.
 Then for all $t\geq 1$,  there are equalities
\begin{itemize}
\item[(1)]$\reg\,(\frac{S}{(I+J)^t})=\max\limits_{\begin{subarray}{c}
i\in [t-1]\\
j\in [t]\\
\end{subarray}} \{\reg\,(\frac{S_1}{I^{t-i}})+\reg\,(\frac{S_2}{J^{i}})+1,\reg\,(\frac{S_1}{I^{t-j+1}})+\reg\,(\frac{S_2}{J^{j}})\}$,
\item[(2)]$\depth\,(\frac{S}{(I+J)^t})=\min\limits_{\begin{subarray}{c}
i\in [t-1]\\
j\in [t]\\
\end{subarray}} \{\depth\,(\frac{S_1}{I^{t-i}})+\depth\,(\frac{S_2}{J^{i}})+1, \depth\,(\frac{S_1}{I^{t-j+1}})+\depth\,(\frac{S_2}{J^{j}})\}$.
\end{itemize}
In particular, If $t=1$, we obtain Lemma  \ref{sum1} $(1)$ and $(3)$.
\end{Lemma}

\begin{Lemma}\label{depthlemma} {\em \cite[Lemmas 1.1 and 1.2]{HTT}} Let\ \ $0\rightarrow A \rightarrow  B \rightarrow  C \rightarrow 0$\ \  be a short exact sequence of finitely generated graded $S$-modules. Then
\begin{itemize}
\item[(1)]$\reg\,(B)\leq \max\,\{\reg\,(A), \reg\,(C)\}$,  the equality holds if $\reg\,(A)-1\neq\reg\,(C)$,
\item[(2)]$\reg\,(C)\leq \max\,\{\reg\,(A)-1, \reg\,(B)\}$, the equality holds if $\reg\,(A)\neq\reg\,(B)$,
\item[(3)]$\depth\,(B)\geq \min\,\{\depth\,(A), \depth\,(C)\}$, the equality holds if $\depth\,(A)-1\neq\depth\,(C)$.
\end{itemize}
\end{Lemma}

\begin{Corollary}\label{add1}
Let  $I\subset S$ be a  nonzero proper homogenous ideal, and   $f$  a homogenous polynomial of degree $d$ in $S$. Then
\begin{itemize}
\item[(1)]$\reg\,(I)\leq \max\,\{\reg\,(I\!:\!f)+d, \reg\,((I,f))\}$,  the equality holds if $\reg\,((I,f))\neq\reg\,(I\!:\!f)+d-1$,
\item[(2)]$\reg\,(I,f)\leq \max\,\{\reg\,(I\!:\!f)+d-1, \reg\,(I)\}$,  the equality holds if $\reg\,(I)\neq\reg\,(I\!:\!f)+d$,
\item[(3)]$\depth\,(S/I)\geq \min\,\{\depth\,(S/(I\!:\!f)), \depth\,(S/(I,f))\}$, the equality holds if $\depth\,(S/(I\!:\!f))\neq\depth\,(S/(I,f))+1$.
\end{itemize}
\end{Corollary}
\begin{proof}It follows from  Lemma \ref{depthlemma} by considering the following short exact sequence
$$0\longrightarrow \frac{S}{I:f}(-d) \stackrel{ \cdot f} \longrightarrow \frac{S}{I}\longrightarrow  \frac{S}{(I,f)}\longrightarrow  0.$$
\end{proof}

\begin{Remark}\label{add2}
Let  $I\subset S$ be a  nonzero proper homogenous ideal and $f_i$ a  homogenous polynomial of degree $d_i$ in $S$ for $i=1,\ldots,m$.
Set $J_0=I$, $J_i=I+(f_1,f_2,\ldots, f_i)$ for any  $i\in [m]$.
Applying the  above corollary, one has
$\reg\,(J_{i-1})\leq \max\{\reg\,(J_{i-1}:f_{i})+d_i,\ \reg\,(J_{i})\}$
and the equality holds if $\reg\,(J_{i})\ge \reg\,(J_{i-1}:f_{i})+d_i$.
Furthermore,
$$\reg\,(I)\leq \max\{\reg\,(J_{j-1}:f_{j})+d_j,\  \reg\,(J_{m})\mid j\in [m]\}$$
and $\reg\,(I)=\reg\,(J_{m})$  if  $\reg\,(J_{m})\geq \max\{\reg\,(J_{j-1}:f_{j})+d_j\mid j\in [m]\}$.

Similarly, we also have
$$\depth\,(S/J_{i-1})\geq \min\{\depth\,(S/(J_{i-1}:f_{i})),\ \depth\,(S/J_{i})\}$$
and the equality holds if $\depth\,(S/(J_{i-1}:f_{i}))\leq \depth\,(S/J_{i})$.
Furthermore,
$$\depth(S/I)\geq \min\{\depth(S/(J_{j-1}:f_{j})), \ \depth(S/J_{m})\mid j\in [m]\}$$
and the equality holds if $\depth(S/(I\!:\!f_{1}))\leq \min\{\depth(S/(J_{j-1}\!:\!f_{j})), \depth(S/J_{m})\mid 2\leq j\leq m\}.$
\end{Remark}

Applying the above  remark  iteratively, we can obtain
\begin{Corollary}\label{add3}
Let $I\subset S$ be a monomial ideal and $u_i$ a monomial for  $i=1,\ldots,r$. Set $V=\{u_1,u_2,\ldots,u_r\}$. Then
\begin{itemize}
\item[(1)]$\reg\,(I)\leq \max\{\reg((I:\prod\limits_{u\in W}u)+I_{W^c})+d_W\mid \ W\subseteq V\}$,
\item[(2)]$\depth\,(S/I)\geq \min\{\depth\,(\frac{S}{(I:\prod\limits_{u\in W}u)+I_{W^c}})\mid \  W\subseteq V\}$,
\end{itemize}
where $W^c=V\setminus W$, $I_{W^{c}}$ is an ideal generated by $\{v\mid v\in W^{c}\}$ and $d_W$ is the degree of $\prod\limits_{u\in W}u$.
\end{Corollary}

\begin{Remark}\label{add4}
Let  $I$ be a  monomial ideal of $S$ with $\mathcal{G}(I)=\{u_1,u_2,\ldots, u_r\}$, where  $u_i$ is a  monomial  of degree $d_i$ for any $i\in[r]$.
Set $J_0=I$, $J_i=(u_{i+1},u_{i+2},\ldots, u_r)$. Thus we have the following short exact sequence
$$0\longrightarrow \frac{S}{J_{i}:u_i}(-d_i) \stackrel{ \cdot u_i} \longrightarrow \frac{S}{J_i}\longrightarrow  \frac{S}{J_{i-1}}\longrightarrow  0.$$
Therefore,
$\reg\,(J_{i-1})\leq \max\{\reg\,(J_{i}:u_{i})+d_i-1,\ \reg\,(J_{i})\}$ by  Lemma \ref{depthlemma} (2),
and the equality holds if $\reg\,(J_{i}:u_{i})+d_i>\reg\,(J_{i})$.
Furthermore,
$$\reg\,(I)=\reg\,(J_0)\leq \max\{\reg\,(J_{j}:u_{j})+d_j-1,\  \reg\,(J_{r})\mid j\in [r]\}$$
and the equality holds if
$\reg\,(J_{1}:u_{1})+d_1>\max\{\reg\,(J_{j}:u_{j})+d_j-1,\  \reg\,(J_{r})\mid 2\leq j\leq r\}.$
\end{Remark}

\section{Irreducible decomposition of the facet ideal  }

In this section, we  provide  the irreducible decomposition of  the facet ideal $\F(\Delta_{m,n})$ of a chessboard complex  $\Delta_{m,n}$ with $n\geq m$.
We first give a definition of the vertex cover.

\begin{Definition} \label{cover}
Let $\Delta$  be a simplicial complex on the vertex set $V$. A {\it  vertex cover} for $\Delta$  is a subset $A$ of $V$ that intersects
every facet of $\Delta$. If $A$ is a minimal element {\em (}under inclusion{\em)} of the set of vertex covers of $\Delta$, it is called a {\it minimal} vertex cover.
 \end{Definition}

We need the following lemma.
\begin{Lemma} \label{prime}{\em (\cite[Proposition 1.8]{F})} Let  $\Delta$  be a simplicial complex, $\F(\Delta)$ its facet ideal.
Then an ideal $P=(x_{i_1},\ldots, x_{i_s})$ is a minimal prime of $\F(\Delta)$ if and only if $\{x_{i_1},\ldots, x_{i_s}\}$ is a minimal vertex cover of $\Delta$.
\end{Lemma}

 Now, we describe the irreducible decomposition of  $\F(\Delta_{m,n})$.
\begin{Theorem} \label{decomposition}
Let $\Delta_{m,n}$ be a chessboard complex with $n\geq m\geq 1$.
Let $V$ be its vertex set and  $\F(\Delta_{m,n})$ be its facet ideal.
 A  prime ideal $P$  of $S$ is a minimal prime of  $\F(\Delta_{m,n})$   if and only if $P$ is  generated by a set $C$,  which is obtained from $V$ by deleting all elements in $s$ rows and all elements in  $m-1-s$ columns for some $0\leq s\leq m-1$.
\end{Theorem}
\begin{proof}The case $m=1$  is trivial. Now we assume that  $m\geq 2$. It is enough to show that the set $C$ is exactly a minimal vertex cover of $\Delta_{m,n}$ by Lemma  \ref{prime}.

First, we show that the set $C$ is a  vertex cover of $\Delta_{m,n}$.

For any $F\in Facets(\Delta_{m,n})$, we know that  $|F|=m$ and all elements in $F$ are taken from different rows and columns.
We may write $F$ as $F=\{x_{1t_1},\ldots,x_{mt_m}\}$, where $t_{1},\ldots,t_{m}\in [n]$ are different.
Now, we  show that there exists some $x_{kt_k}\in F$ such that $x_{kt_k}\in C$, i.e., $F\cap C\neq \emptyset$. Assume that the set $C$ is obtained from $V$ by deleting all elements in  $i_1,\ldots,i_{s}$ rows and all elements in $j_1,\ldots,j_{m-1-s}$ columns, where $\{i_1,\ldots,i_{s}\}\subsetneq [m]$ and $\{j_1,\ldots,j_{m-1-s}\}\subsetneq [n]$.
By convention, $\{i_1,\ldots,i_{s}\}=\emptyset$ if $s=0$, and $\{j_1,\ldots,j_{m-1-s}\}=\emptyset$ if $s=m-1$.
Note that there exist $m-s$ elements in $F$, which are not in the $i_1,\ldots,i_{s}$ rows, and at least one of such $m-s$ elements is not in the  $j_1,\ldots,j_{m-1-s}$ columns since $m-s>m-1-s$. Say
$x_{kt_k}$. Hence $F\cap C\neq \emptyset$, as desired.

Next, we show that the vertex cover $C$ is  minimal.

Let $C$ be a set  obtained from $V$ by deleting all elements in  $i_1,\ldots,i_{s}$ rows and all elements in $j_1,\ldots,j_{m-1-s}$ columns, where $i_1,\ldots,i_{s}\in[m]$ and $j_1,\ldots,j_{m-1-s}\in [n]$ are different. Let  $W$ be any proper subset of $C$. Then there exists some $x_{pq}\in C\setminus W$.  Hence
$p\in [m]\setminus \{i_1,\ldots,i_{s}\}$, $q\in [n]\setminus \{j_1,\ldots,j_{m-1-s}\}$. Choose $(m-1-s)$ different elements $k_1,\ldots,k_{m-1-s}\in [m]\setminus \{p,i_1,\ldots,i_{s}\}$ and $s$ different  elements
 $\ell_1,\ldots,\ell_{s}\in [n]\setminus \{q,j_1,\ldots,j_{m-1-s}\}$. Set $F=\{x_{i_1\ell_1},\ldots,x_{i_s\ell_s},x_{k_1j_1},\ldots,x_{k_{m-1-s}j_{m-1-s}},
x_{pq}\}$.  Note that $|F|=m$ and all elements in $F$ are taken from different rows and columns, it follows that $F\in Facets(\Delta_{m,n})$ and $C\cap\{x_{i_1\ell_1},\ldots,x_{i_s\ell_s},x_{k_1j_1},\ldots,x_{k_{m-1-s}j_{m-1-s}}\}=\emptyset$. By the choice of $x_{pq}$, one has
 $F\cap W=\emptyset$. This implies that $W$  is not a  vertex cover of $\Delta_{m,n}$.
\end{proof}

\medskip
Given an ideal $I\subset S$, we set
 $$\bight\,(I)=\sup\{\height\,(P)\mid  P \text{ is a   minimal prime  ideal  of } S \text{ over }  I\}.$$
As a consequence of the above theorem, we have

 \begin{Corollary}\label{dimension}Let $\Delta_{m,n}$ be  a chessboard complex with $n\geq m\geq 1$, one has
\begin{enumerate}
 \item[(a)] $\height\,(\F(\Delta_{m,n}))=n$,
 \item[(b)] $\dim\,(S/\F(\Delta_{m,n}))=(m-1)n$,
 \item[(c)] $\bight\,(\F(\Delta_{m,n}))=\left
\{\begin{array}{l@{\  \ }l}
\lfloor\frac{(n+1)^2}{4}\rfloor&\text{if   $n<2m-1$,}\\
{(n-m+1)m}&\text{otherwise.}
\end{array}\right.$
\end{enumerate}
 \end{Corollary}
  \begin{proof}
(a)  follows from Theorem \ref{decomposition}.

 (b) follows from a fact that $\height (\F(\Delta_{m,n}))+\dim (S/\F(\Delta_{m,n}))=\dim (S)=mn$.

(c) The case  $m=1$ is trivial.
Now suppose $m\geq 2$.  Let $P$ be a minimal prime of  $\F(\Delta_{m,n})$, then  its generators  have  forms as the above theorem. It follows that
$$\height\,(P)=[n-(m-1-s)](m-s)=-[s+\frac{n-(2m-1)}{2}]^2+\frac{(n+1)^2}{4}$$
for some nonnegative integer $s\leq m-1$. Therefore,
$$
\bight\,(\F(\Delta_{m,n}))=\sup\{-[s+\frac{n-(2m-1)}{2}]^2+\frac{(n+1)^2}{4}\mid \,0\leq s\leq m-1\}.
$$
Let $f(s)=-[s+\frac{n-(2m-1)}{2}]^2+\frac{(n+1)^2}{4}$, then $f(s)$ is a quadratic function in $s$. Hence, if $n\geq 2m-1$, then   $$\bight\,(\F(\Delta_{m,n}))=f(0)=-[\frac{n-(2m-1)}{2}]^2+\frac{(n+1)^2}{4}=(n-m+1)m;$$
if $n<2m-1$, then
\begin{eqnarray*}\bight\,(\F(\Delta_{m,n}))&=&f(\lfloor\frac{(2m-1)-n}{2}\rfloor)\\
&=&-[\lfloor\frac{(2m-1)-n}{2}\rfloor+\frac{n-(2m-1)}{2}]^2+\frac{(n+1)^2}{4}
=\lfloor\frac{(n+1)^2}{4}\rfloor,\end{eqnarray*}
 where the last equality holds because $\bight\,(\F(\Delta_{m,n}))=\frac{(n+1)^2}{4}$ if $n$ is odd, otherwise,  $\bight\,(\F(\Delta_{m,n}))=\frac{(n+1)^2}{4}-\frac{1}{4}$. The proof is complete.
\end{proof}

\section{ Powers of  facet ideals of  small chessboard complexes }
In this section, we  give some formulas for depth  and  regularity of  powers of the facet ideal $\F(\Delta_{m,n})$ of a chessboard complex  $\Delta_{m,n}$ under the condition that $n\ge m\geq 1$ and $m\leq 2$.

Let $G=(V(G),E(G))$ be a finite simple graph with the vertex set $V(G)$ and edge set $E(G)$. We say that a graph $H$ is an {\em induced subgraph} of $G$ on a subset $V'$ of $V(G)$ if $V(H)=V'$ and $E(H)=\{uv\in E(G) : u,v \in V'\}$. The {\it complement} of $G$ is the graph $G^{c}$ with $V(G^{c})=V(G)$ and $E(G^{c})=\{uv :u,v \in  V(G) \text{\ and\ } uv \notin E(G)\}$.  Let $u\in  V(G)$, the degree of $u$, denoted by $\deg\,(u)$, is the number of  vertices of $G$ adjacent to $u$.
Let $G=(V(G),E(G))$ be a connected graph with  $n=|V (G)|$ vertices, it is called an $n$-cycle if $\deg\,(u)=2$ for any $v\in V(G)$,
denoted by $C_n$, where $n$ is referred to as the length of the cycle.

 A  bipartite graph $G$ is a graph whose  vertex set $V(G)$ is the disjoint union $X\sqcup Y$ of $X$ and $Y$ and edge set $E(G)$ satisfies $E(G)\subseteq \{xy\mid x\in X, y\in Y\}$.
The bipartite complement of a graph $H$ is the bipartite graph
$H^{bc}$ with $V(H^{bc})=X\sqcup Y$ and $E(H^{bc})=\{xy: x\in X, y\in Y \text{\ and\ } xy\notin E(H)\}$.

\begin{Lemma}\label{three}{\em (\cite[Theorem 3.6]{AB})}
Let $G$ be a connected bipartite graph and  $I(G)$ its edge ideal. If $\reg\,(I(G))=3$, then
$\reg\,(I(G)^t)=2t+1$ for any $t\geq 1$.
 \end{Lemma}

\begin{Theorem}\label{reg}
Let $\Delta_{1,n}$ be a chessboard complex with $n\geq 1$, then for any $t\geq 1$,
\begin{enumerate}
 \item[(a)] $S/\F(\Delta_{1,n})^t$ is Cohen-Macaulay of dimension $0$.
 \item[(b)] $\reg\,(S/\F(\Delta_{1,n})^t)=t-1$.
\end{enumerate}
In particular, $\F(\Delta_{1,n})^t$ has a linear resolution. \end{Theorem}

 \begin{proof}  (a)  is trivial,  since $\F(\Delta_{1,n})=(x_{11},\ldots,x_{1n})$ is the maximal graded ideal of $S$.
(b) follows from Lemmas \ref{quotient} and  \ref{sum},  since $\F(\Delta_{1,n})=(x_{11},\ldots,x_{1n})$ is a complete intersection ideal.
 \end{proof}

\medskip
 Let $I$ be a  squarefree monomial ideal in a polynomial ring $k[x_1,\ldots, x_n]$. The {\it Stanley-Reisner complex}  $\delta(I)$ of $I$ is a simplicial  complex on the vertex set $\{x_1,\ldots, x_n\}$, where
$\{x_{i_1},\ldots, x_{i_s}\}$ is a face of $\delta(I)$ if and only if $x_{i_1}\cdots x_{i_s}\notin I$.

Let $\Delta_{2,n}$ be a chessboard complex with $n\geq 3$, then its facet ideal $\F(\Delta_{2,n})=(x_{1i}x_{2j}\mid 1\leq i,j\leq n \text{\ and }i\neq j)$ can be regarded as the edge ideal of a connected bipartite graph  with bipartition $\{x_{11},\ldots,x_{1n}\}\sqcup \{x_{21},\ldots,x_{2n}\}$, and the  Stanley-Reisner complex $\delta(\F(\Delta_{2,n}))$ of  $\F(\Delta_{2,n})$ is
$$
\delta(\F(\Delta_{2,n})=\big\langle\, \{x_{11},\ldots,x_{1n}\},\{x_{21},\ldots,x_{2n}\},\{x_{11},x_{21}\},\{x_{12},x_{22}\},\ldots,\{x_{1n},x_{2n}\}\,\big\rangle
$$
from Theorem  \ref{decomposition} and \cite[Remark 1.13]{F}.
By \cite[Lemma 3.1]{T}, we have

\begin{Lemma}\label{power} Let  $\Delta_{2,n}$ be a chessboard complex with $n\geq 3$. Then
 $\depth\,(S/\F(\Delta_{2,n})^t)\\ =1$ if and only  if $\Delta_{\alpha}(\F(\Delta_{2,n})^t)=\langle F_1,F_2\rangle$  for  some $\alpha=(\alpha_1,\ldots,\alpha_{2n})\in \Nb^{2n}$, $F_1=\{x_{11},\ldots,x_{1n}\}$,  $F_2=\{x_{21},\ldots,x_{2n}\}$  and $\Delta_{\alpha}(\F(\Delta_{2,n})^t)=\big\langle\, F\in Facets\,(\delta(\F(\Delta_{2,n}))) \mid \sum\limits_{i\notin F}\alpha_i\leq t-1\,\big\rangle$.
 Moreover,  if $t=\min\,\{k\mid \depth\,(S/\F(\Delta_{2,n})^{k})=1\}$, then such $\alpha$ must satisfy
$\sum\limits_{i\notin F_1}\alpha_i=\sum\limits_{i\notin F_2}\alpha_i=t-1.$
\end{Lemma}

\begin{Theorem}\label{depth}
Let $\Delta_{2,n}$ be a chessboard complex with $n\geq 2$, then for any $t\geq 1$, we have
\begin{enumerate}
 \item[(a)] $\reg\,(S/\F(\Delta_{2,n})^t)=2t$.
 \item[(b)] $\depth\,(S/\F(\Delta_{2,n})^t)=1$ or $2$. Moreover, $\depth\,(S/\F(\Delta_{2,n})^t)=1$ if and only  if  $n=3$ and
     $t\geq 4$, or $n\geq 4$ and  $t\geq 3$.
 \item[(c)] $S/\F(\Delta_{2,2}^t)$ is Cohen-Macaulay of dimension $2$.
\end{enumerate}
\end{Theorem}
\begin{proof}
(a)  $\F(\Delta_{2,n})$ can be regarded as the edge ideal of a connected bipartite graph $G$ with bipartition $\{x_{11},\ldots,x_{1n}\}\sqcup \{x_{21},\ldots,x_{2n}\}$. The bipartite complement $G^{bc}$ of $G$ contains
an induced cycle $C_4$ whose edges are $x_{11}x_{12},x_{11}x_{21},x_{21}x_{22},x_{12}x_{22}$
and  does not contain  the induced circle of length $\geq 6$, since $x_{1i}x_{2i}\in E(G^{c})$ for any $1\leq i\leq n$.
Hence $\reg\,(\F(\Delta_{2,n}))=3$ by \cite[Theorem 3.1]{FG}, the desired result follows from Lemmas \ref{three} and \ref{quotient}.

(b) $\F(\Delta_{2,n})$ is normally torsion free and  $\min\{\depth\,(S/\F(\Delta_{2,n})^t)\mid t\geq 1\}=1$ by  \cite[Theorem 5.9]{SVV} and \cite[Theorem  4.4]{T}, since  $\F(\Delta_{2,n})$  is the edge ideal of a connected bipartite graph. By \cite[Proposition 1.2.13]{BH} and Theorem \ref{decomposition}, we obtain
\begin{eqnarray*}
 \depth(S/\F(\Delta_{2,n})^t)&\leq& \min\{\dim(S/P): P\in  \Ass(\F(\Delta_{2,n})^t)\}\\
 &=& \min\{\dim(S/P): P\in  \Ass(\F(\Delta_{2,n}))\}=2.
 \end{eqnarray*}
Therefore, $\depth\,(S/\F(\Delta_{2,n})^t)=1$ or $2$.

If $n=2$, then  $\F(\Delta_{2,2})=(x_{11}x_{22},x_{12}x_{21})$ is a  complete intersection ideal. By Lemma \ref{sum}, we have $\depth\,(S/\F(\Delta_{2,2})^t)=2$ for any $t\geq 1$;
If $n=3$, then $\F(\Delta_{2,n})$  is the edge ideal of a cycle graph of length  $6$.  It follows that  $\depth\,(S/\F(\Delta_{2,n})^t)=1$ if and only  if $t\geq 4$  by \cite[Lemma 5.3]{T}.

Now suppose  $n\geq 4$. Choose $\alpha=(\alpha_1,\ldots,\alpha_{2n})\in \Nb^{2n}$, where $$\alpha_i=\left
\{\begin{array}{l@{\  \ }l}
{1}&\text{if $i=1,2,n+3,n+4$}\\
{0}&\text{otherwise},\\
\end{array}\right.$$
one has   $\Delta_{\alpha}(\F(\Delta_{2,n})^t)=\langle F_1,F_2\rangle$ for all $t\geq 3$, where $F_1=\{x_{11},\ldots,x_{1n}\}$ and $F_2=\{x_{21},\ldots,x_{2n}\}$.

If   $\Delta_{\beta}(\F(\Delta_{2,n})^2)=\langle F_1,F_2\rangle$ for some  $\beta=(\beta_1,\ldots,\beta_{2n})\in \Nb^{2n}$, then
$\sum\limits_{i=1}^{n}\beta_i\leq 1$ because of  $F_2\in \Delta_{\beta}(\F(\Delta_{2,n})^2)$. Similarly, $\sum\limits_{i=n+1}^{2n}\beta_i\leq 1$.
Therefore, there exists at most one $\beta_i=1$ for  $1\leq i\leq n$. Thus $\{x_{1i},x_{2i}\}\in \Delta_{\beta}(\F(\Delta_{2,n})^2)$,  contradicting with the supposition $\Delta_{\beta}(\F(\Delta_{2,n})^2)=\langle F_1,F_2\rangle$.
It follows  from Lemma \ref{power}  that $\depth\,(S/\F(\Delta_{2,n})^t)=1$ if and only  if    $t\geq 3$.

(c) Since $\F(\Delta_{2,2})=(x_{11}x_{22},x_{12}x_{21})$ can be regarded as the edge ideal of a bipartite graph,
and its symbolic powers are consistent with ordinary powers.
Hence $\F(\Delta_{2,2})^t=(x_{11}x_{22},x_{12}x_{21})^t=(x_{11},x_{12})^t\cap (x_{21},x_{22})^t\cap(x_{11},x_{21})^t\cap (x_{21},x_{22})^t$ for any $t\geq 1$. It follows that $\dim\,(S/\F(\Delta_{2,2})^t)=2$. The result follows from (b).
\end{proof}

\section{\text{The facet ideal $\F(\Delta_{m,n})$ with $m\geq 3$}}
In this section, we  provide some lower bounds for   depth  and  regularity of the facet ideal $\F(\Delta_{m,n})$ of a chessboard complex  $\Delta_{m,n}$ under the condition that $n\ge m\ge 3$,
  we also prove that these  lower bounds can be attained if  $m=3$.

Let $\Delta$ be a simplicial complex on the vertex set $V$, $Facets\,(\Delta)$
 the set  of its facets. A subset $M$ of $Facets\,(\Delta)$ is called a {\it matching} of $\Delta$ if the facets in $M$ are pairwise disjoint.
Moreover, if $Facets\,(\Delta|_A)=M$ for $A=\cup_{F\in M}F$, then $M$ is called an {\it induced matching} of $\Delta$.
\begin{Lemma}\label{matching}{\em (\cite[Corollary 3.5]{EF} and \cite[Corollary 3.9]{MV})}
 Let $\Delta$ be a simplicial complex with facet ideal $\F(\Delta)$.
 Then
$$
\reg\,(S/\F(\Delta))\geq \Big\{|\bigcup\limits_{i=1}^{k}F_i|-k\mid \{F_1,\ldots,F_k\} \text{\  is an induced matching in } \Delta \Big\}.
$$
\end{Lemma}

\begin{Corollary}\label{chessbord1}
Let $\Delta_{m,n}$ be a chessboard complex with $n\geq m\geq 3$, then
$$
\reg(S/\F(\Delta_{m,n}))\geq 2(m-1).
$$
 \end{Corollary}
 \begin{proof}
 Let $F_1=\{x_{11},x_{22},\ldots,x_{mm}\}$ and $F_2=\{x_{12},x_{23},\ldots,x_{(m-1)m},x_{m1}\}$, then
$\{F_1, F_2\}$ is an induced matching of $\Delta_{m,n}$. The result follows from  Lemma \ref{matching}.
\end{proof}

\begin{Theorem}\label{facetcover}
Let $\Delta_{3,n}$ be a chessboard complex with $n\geq 3$, then
$$\reg\,(S/\F(\Delta_{3,n}))=4.$$
 \end{Theorem}
 \begin{proof} It is clear that $\reg\,(S/\F(\Delta_{3,n}))\geq 4$ by  Corollary \ref{chessbord1}, so it is enough to show
that $\reg\,(\F(\Delta_{3,n}))\leq 5$ by Lemma \ref{quotient} (3). We sort all the elements of  $\mathcal{G}(\F(\Delta_{2,n}))$ in lexicographic order
  with $x_{11}>x_{12}>\cdots>x_{1n}>x_{21}>x_{22}>\cdots>x_{2n}$, that is :
$$x_{11}x_{22}>x_{11}x_{23}>\cdots>x_{11}x_{2n}>x_{12}x_{21}>x_{12}x_{23}>\cdots>x_{1n}x_{21}>\cdots>x_{1n}x_{2,n-1}.$$
Let $u_p$ be the $p$-th element in  $\mathcal{G}(\F(\Delta_{2,n}))$, where $1\leq p\leq r$ and $r=|\mathcal{G}(\F(\Delta_{2,n}))|$.
Put  $J_0=\F(\Delta_{3,n})$, $Q_p=(u_1,u_2,\ldots,u_p)$ and $J_p=J_0+Q_p$ for any $ p\in [r]$.
By Remark \ref{add2}, we have
$$\reg\,(\F(\Delta_{3,n}))=\reg\,(J_0)\leq \max\{\reg\,(J_{p}:u_{p+1})+2,\  \reg\,(J_{r})\mid 0\leq p<r\}.\eqno(1)$$
Note that $\reg\,(J_r)=\reg\,(\F(\Delta_{2,n}))=3$ by Theorem \ref{depth}.
It is enough  to show that  $\reg\,(J_{p}:u_{p+1})\leq 3$ for any $0\leq p<r$ by formula (1). Let $u_{p+1}=x_{1i}x_{2j}$ with $1\leq i,j\leq n$ and $i\neq j$. By a direct calculation, we obtain that
$$J_0:u_{p+1}=\F(\Delta_{3,n}):x_{1i}x_{2j}=K_3+(x_{3i})(K_1+(x_{1j})K_2)+(x_{3j})(K_2+(x_{2i})K_1),\eqno(2)$$
and
$$
Q_p:u_{p+1}=Q_p:x_{1i}x_{2j}=\left\{\begin{array}{ll}
L_2\ \ &\text{if}\ \ i=1\ \text{and}\ j>1,\\
L_1+L_2\ \ &\text{if} \ \ j>i\geq 2,\\
L_1+(x_{11})L_3\ \ &\text{if}\ \ j=1\ \text{and}\ i>1,\\
L_1+L_2+(x_{1j})L_3\ \ &\text{if} \ \ i>j\geq 2,
\end{array}\right.
\eqno(3)$$
where $\mathcal {G}(K_{\ell})=\{x_{\ell1},\ldots, x_{\ell m}\}\setminus \{x_{\ell i},x_{\ell j}\}$ for $\ell\in [3]$, $\mathcal {G}(L_1)=\{x_{11},\ldots, x_{1(i-1)}\}\setminus \{x_{1j}\}$, $\mathcal {G}(L_2)=\{x_{21},\ldots, x_{2(j-1)}\}\setminus \{x_{2i}\}$ and $\mathcal {G}(L_3)=\{x_{2(j+1)},\ldots, x_{2n}\}$.
By Lemma \ref{depthreglemma} and formulas (2), (3), one has
\begin{eqnarray*}
\hspace{1cm} \mbox{reg}\,(J_p:u_{p+1})&=&\mbox{reg}\,((J_0:u_{p+1})+(Q_p:u_{p+1}))\\
&\leq&\left\{\begin{array}{ll}
\mbox{reg}\,(J_0:u_{p+1})\ \ &\text{if}\ \ i<j,\\
\mbox{reg}\,((J_0:u_{p+1})+(x_{1j})L_3)\ \ &\text{if} \ \ i>j.
\end{array}\right.\hspace{2cm}(4)
\end{eqnarray*}
We distinguish into the following two cases:

(i) If $i<j$. By formula (2), one has
$$J_0:u_{p+1}x_{3j}=K_3+(x_{3i})(K_1+(x_{1j})K_2)+K_2+(x_{2i})K_1=K_3+K_2+(x_{2i},x_{3i})K_1$$
 and $((J_0:u_{p+1}),x_{3j})=K_3+(x_{3j})+(x_{3i})(K_1+(x_{1j})K_2)$.  By Lemma  \ref{sum1} (3), (4), we get that
$\reg\,(J_0:u_{p+1}x_{3j})=\reg\,((x_{2i},x_{3i})K_1)=\reg\,((x_{2i},x_{3i}))+\reg\,(K_1)=2$
and
\begin{eqnarray*}
\reg\,((J_0:u_{p+1}),x_{3j})&=&\reg\,(K_3+(x_{3j})+(x_{3i})(K_1+(x_{1j})K_2))\\
&=&\reg\,((x_{3i})(K_1+(x_{1j})K_2))=\reg\,((x_{3i}))+\reg\,(K_1+(x_{1j})K_2)\\
&=&1+\reg\,((x_{1j})K_2)=3.
\end{eqnarray*}
Hence, using Corollary \ref{add1}, we have
$$\reg\,(J_p\!:\!u_{p+1})\le\reg\,(J_0\!:\!u_{p+1})\leq \max\{\reg\,(J_0\!:u_{p+1}x_{3j})+1,\reg\,((J_0\!:\!u_{p+1}),x_{3j})\}=3.$$

(ii) If $i>j$. Let $L=(J_0:u_{p+1})+(x_{1j})L_3$ for any $0\leq p\leq r-1$. Using Corollary \ref{add1} repeatedly, we obtain
\begin{eqnarray*}
\reg\,(J_p\!:\!u_{p+1})&\le&\reg(L)\leq \max\{\reg(L:x_{3i})+1, \reg(L,x_{3i}) \}\\
&\leq& \max\{\reg(L:x_{3i}x_{1j})+2, \reg(L:x_{3i},x_{1j})+1, \reg((L:x_{3j}),x_{3i})+1,\\
&&\ \ \ \ \ \  \reg(L,x_{3i},x_{3j})\}. \hspace{7.0cm}(5)
\end{eqnarray*}
Note that
\begin{eqnarray*}
(L:x_{3i}x_{1j})&=&K_3+K_1+K_2+L_3,\\
((L:x_{3i}),x_{1j})&=&K_3+K_1+(x_{3j})K_2+(x_{1j}),\\
((L:x_{3j}),x_{3i})&=&K_3+K_2+(x_{2i})K_1+(x_{1j})L_3+(x_{3i})\\
&=&K_3+K_2+(x_{2i})(K_1,x_{1j})+(x_{3i}),\\
(L,x_{3i},x_{3j})&=&K_3+(x_{1j})L_3+(x_{3i},x_{3j}).
\end{eqnarray*}
 By Lemma \ref{sum1} (3) and (4),  we obtain  $\mbox{reg}\,(L:x_{3i}x_{1j})=1$ and
$$\mbox{reg}\,(((L:x_{3i}),x_{1j}))=\mbox{reg}\,(((L:x_{3j}),x_{3i}))=\mbox{reg}\,((L,x_{3i},x_{3j}))=2.$$
Thus, by formula (5), we obtain that $\reg\,(J_p:u_{p+1})\leq 3$.\end{proof}

\medskip
Given a chessboard complex $\Delta_{m,n}$  with $n\geq m\geq 2$.  For $1\leq i\leq n$, we define $A_{m,i}$ to be a subcomplex of $\Delta_{m,n}$,
the set of its facets is $Facets\,(A_{m,i})=\{F\in Facets\,(\Delta_{m,n})\mid x_{ms}\notin F \text{\ \ for any\ } 1\leq s\leq i\}$
and $B_{m,i}$ to be a chessboard complex  obtained by removing   the $m$-th row cells and $i$-th column cells of of the original chessboard. It is obvious that $A_{m,n}=\emptyset$.

\begin{Example}
For a chessboard complex  $\Delta_{2,3}$,  $A_{2,2}$ is a subcomplex of $\Delta_{2,3}$ with  facets  $\{x_{11},x_{23}\}$ and
$\{x_{12},x_{23}\}$. $B_{2,2}$ is also a chessboard complex with
 facets $\{x_{11}\}$ and $\{x_{13}\}$.
\end{Example}

 \begin{Lemma}\label{depth1}
Let $\Delta_{m,n}$ be a chessboard complex and $\F_0:=\F(\Delta_{m,n})$  its facet ideal. For any $1\leq i\leq n$, let $\F_i$ and $G_i$ be  facet ideals of $A_{m,i}$ and  $B_{m,i}$, respectively. Then
\begin{enumerate}
 \item[(a)] $\F_n=(0)$ and $\sum\limits_{i=1}^{n}G_i=\F(\Delta_{m-1,n})$, where $\Delta_{m-1,n}$ is a chessboard complex obtained by removing   the $m$-th row cells  of $\Delta_{m,n}$,
 \item[(b)] $(\F_{i-1}:x_{mi})=\F_i+\G_{i}$ and $(\F_{i-1}, x_{mi})=(\F_i,x_{mi})$ for any $i=1,\ldots,n$.
\end{enumerate}
 \end{Lemma}
\begin{proof}
(a) is obvious. Since  $\F_{i-1}=\sum\limits_{j=i}^{n}x_{mj}\G_j$ for $1\leq i\leq n$,  (b)  follows.
\end{proof}

 \begin{Remark}\label{depth2}
Let $\Delta_{m,n}$ be a chessboard complex and $\F_0:=\F(\Delta_{m,n})$ be its facet ideal. Let $V=\{x_{m1},x_{m2},\ldots,x_{mn}\}$,  $W\subseteq V$ and
$\F_{W}=(\F_0:\prod\limits_{x_{mi}\in W}x_{mi})+P_{W^c}$. Using the lemma above repeatedly, we obtain
$$\F_{W}=\sum\limits_{x_{mi}\in W}G_i+P_{W^c}$$
where  $P_{W^c}$ is an ideal generated by $W^c$ and $W^c$ is the complement set of $W$ in  $V$.
\end{Remark}

Given a chessboard complex $\Delta_{m,n}$ with $n\geq m$. For $1\leq i_1<\cdots <i_m\leq n$,
we define a subcomplex $\D_{i_1,\ldots,i_m}$ of $\Delta_{m,n}$,
the set of its facets is $Facets\,(\D_{i_1,\ldots,i_m})=\{F\in Facets\,(\Delta_{m,n})\mid F\subseteq \bigcup\limits_{\ell=1}^{m}\{x_{1i_\ell},x_{2i_\ell},\ldots,x_{mi_\ell}\}\}$.
In fact, it is a chessboard complex on an $ m\times m$ chessboard composed of the $i_1$-th, $\dots, i_m$-th  columns of the original chessboard.

\begin{Theorem}\label{depth4}
Let $\Delta_{m,n}$ be a chessboard complex  with $n\geq m$ and $$\Omega=\{\D_{i_1,\ldots,i_m}\mid 1\leq i_1<\cdots <i_m\leq n\}.$$
Then
$$
\depth\,\left(\frac{S}{\sum\limits_{\D_{i_1,\ldots,i_m}\in \Omega'}\F(\D_{i_1,\ldots,i_m})}\right)\geq 2(m-1)
$$
for any  nonempty subcollection $\Omega'$ of  $\Omega$.
\end{Theorem}
\begin{proof}
We apply induction on $m$. Case $m=1$ is trivial. Now suppose  $m\geq 2$.
 Let $V=\{x_{m1},x_{m2},\ldots,x_{mn}\}$,  $W\subseteq V$ and denote
$$\F_{W}=(\sum\limits_{\D_{i_1,\ldots,i_m}\in \Omega'}\F(\D_{i_1,\ldots,i_m}):\prod\limits_{x_{m\ell}\in W}x_{m\ell})+ P_{W^c}$$
where $P_{W^{c}}$ is an ideal generated by $V\setminus W$.
By Corollary \ref{add3}, it is enough to show that
$$\depth\,(S/\F_{W})\geq 2(m-1).$$

If $W=\emptyset$, then $\F_{W}=\sum\limits_{\D_{i_1,\ldots,i_m}\in \Omega'}\F(\D_{i_1,\ldots,i_m})+P_{W^{c}}=P_{W^{c}}=(x_{m1},\ldots,x_{mn})$. In this case,
$\depth\,(S/\F_W)=\depth\,(S/(x_{m1},\ldots,x_{mn}))=n(m-1)\geq 2(m-1).$

If $W\neq\emptyset$. We may assume $W=\{x_{ms_1},x_{ms_2},\ldots,x_{ms_k}\}$
and
$$\Omega''=\{\D_{i_1,\ldots,i_m}\in \Omega'\mid \{i_1,\ldots,i_m\}\cap\{s_1,\ldots,s_k\}\neq \emptyset\}.$$
We consider the following two cases:

\vspace{3mm}
(a) If $\Omega''= \emptyset$, then $ \{i_1,\ldots,i_m\}\cap\{s_1,\ldots,s_k\}=\emptyset$ for any $\D_{i_1,\ldots,i_m}\in \Omega'$. It follows that
$$(\F(\D_{i_1,\ldots,i_m}):\prod\limits_{x_{m\ell}\in W}x_{m\ell})=\F(\D_{i_1,\ldots,i_m})\subset P_{W^c}$$
because of  $\{x_{mi_1},\ldots,x_{mi_m}\}\subseteq W^c$. Thus
$$\F_W=(\sum\limits_{\D_{i_1,\ldots,i_m}\in \Omega'}\F(\D_{i_1,\ldots,i_m}):\prod\limits_{x_{m\ell}\in W}x_{m\ell})
+ P_{W^c}=P_{W^c}.$$
Thus, we have $\depth\,(S/\F_W)=\depth\,(S/P_{W^c})=n(m-1)+k> 2(m-1)$.

(b)  If $\Omega''\neq \emptyset$, then
\begin{eqnarray*}
\F_W\!\!&=&\!\!(\sum\limits_{\D_{i_1,\ldots,i_m}\in \Omega'}\F(\D_{i_1,\ldots,i_m}):\prod\limits_{x_{m\ell}\in W}x_{m\ell})
+ P_{W^c}\\
\!\!&=&\!\!\!\!\sum\limits_{\D_{i_1,\ldots,i_m}\in \Omega''}\!\!\!(\F(\D_{i_1,\ldots,i_m}):\!\!\prod\limits_{x_{m\ell}\in W}\!\!x_{m\ell})+\!\!\!\!\sum\limits_{\D_{i_1,\ldots,i_m}\in \Omega'\setminus \Omega''}(\F(\D_{i_1,\ldots,i_m}):\!\!\prod\limits_{x_{m\ell}\in W}\!\!x_{m\ell})+P_{W^c}\\
\!\!&=&\!\!\!\sum\limits_{\D_{i_1,\ldots,i_m}\in \Omega''}\!\!\!(\F(\D_{i_1,\ldots,i_m})\!:\!\!\prod\limits_{x_{m\ell}\in W}x_{m\ell})+P_{W^c}\\
\!\!&=&\!\!\!\sum\limits_{\D_{i_1,\ldots,i_m}\in \Omega''}\!\sum\limits_{s_{\ell}\in A_{i_1,\ldots,i_m}}\!\!\!\!\!\F(\B_{i_1,\ldots,i_m,s_{\ell}})\!+\!P_{W^c}\\
\!\!&=&\F'_W+P_{W^c},
\end{eqnarray*}
where the third and fourth equalities  hold because of $(\F(\D_{i_1,\ldots,i_m}):\prod\limits_{x_{m\ell}\in W}x_{m\ell})=\F(\D_{i_1,\ldots,i_m})\subset P_{W^c}$ for any $\D_{i_1,\ldots,i_m}\in \Omega'\setminus \Omega''$ and Remark \ref{depth2}, where $A_{i_1,\ldots,i_m}=\{i_1,\ldots,i_m\}\cap\{s_1,\ldots,s_k\}$,
$\B_{i_1,\ldots,i_m,s_{\ell}}$ is a chessboard complex obtained by removing the $m$-th row cells and $s_{\ell}$-th column cells of the chessboard corresponding to $\D_{i_1,\ldots,i_m}$ and
$$\F'_W=\sum\limits_{\D_{i_1,\ldots,i_m}\in \Omega''}\sum\limits_{s_{\ell}\in A_{i_1,\ldots,i_m}}\F(\B_{i_1,\ldots,i_m,s_{\ell}}).$$
We distinguish into  the following two cases:

(i) If $k=1$, then  $A_{i_1,\ldots,i_m}=\{s_1\}$  for any $\D_{i_1,\ldots,i_m}\in \Omega''$. Thus there exists some $j\in [m]$ such that  $i_j=s_1$.  In this case,  $\F'_W=\sum\limits_{\D_{i_1,\ldots,i_m}\in \Omega''}\F(\B_{i_1,\ldots,i_m,s_{\ell}})$ and $\B_{i_1,\ldots,i_m,s_1}$ is a subcomplex of $\Delta_{m-1,n-1}$
with $Facets\,(\B_{i_1,\ldots,i_m,s_1})=\{F\in Facets\,(\Delta_{m-1,n-1})\mid F\subset \bigcup\limits_{\begin{subarray}{c}
j=1\\
i_j\neq s_1\\
\end{subarray}}^{m}\{x_{1i_j},x_{2i_j},\ldots,x_{mi_j}\}\}$,
whereas $\Delta_{m-1,n-1}$ is a chessboard complex obtained by deleting the $m$-th row cells and the $s_1$-th column cells of the $m\times n$ chessboard.
By induction hypothesis, we obtain $$\depth\,(S/\F_W)=\depth\,(S_1/\F'_W)+m\geq 2(m-2)+m\geq 2(m-1)$$
where $S_1=k[x_{ij}\mid i\in[m-1]  \text{ and\ }j\in[n]\setminus \{s_1\} ]$.

(ii) If $k\geq 2$, then for any $s_{\ell}\in A_{i_1,\ldots,i_m}$, all $\B_{i_1,\ldots,i_m,s_{\ell}}$ are subcomplexes of $\Delta_{m-1,n}$
with $Facets\,(\B_{i_1,\ldots,i_m,s_{\ell}})=\{F\in Facets\,(\Delta_{m-1,n})\mid F\subset \bigcup\limits_{\begin{subarray}{c}
j=1\\
i_j\neq s_\ell\\
\end{subarray}}^{m}\{x_{1i_j},x_{2i_j},\ldots,x_{mi_j}\}\}$, where $\Delta_{m-1,n}$ is a  chessboard complex obtained by removing the $m$-th row cells  of the $m\times n$ chessboard.
By induction hypothesis, we obtain
$$\depth\,(S/\F_W)=\depth\,(S_2/\F'_W)+k\geq 2(m-2)+k\geq 2(m-1)$$
where $S_2=k[x_{ij}\mid i\in[m-1]  \text{ and\ } j\in[n]]$.
\end{proof}

\begin{Corollary}
Let $\Delta_{m,n}$ be a chessboard complex with $n\geq m\geq 1$, then
  $$\depth\,(S/\F(\Delta_{m,n}))\geq 2(m-1).$$
 \end{Corollary}
 \begin{proof} Since $\F(\Delta_{m,n})$ can be expressed as
 $\F(\Delta_{m,n})=\sum\limits_{\D_{i_1,\ldots,i_m}\in \Omega}\F(\D_{i_1,\ldots,i_m})$
where $\Omega=\{\D_{i_1,\ldots,i_m}\mid 1\leq i_1<\cdots <i_m\leq n\}$. The desired result
follows from Theorem \ref{depth4}.
\end{proof}

\medskip
Let $G$ be a simple graph on the set $V=\{x_1,\ldots,x_n\}$, and let  $m\le n$ be an integer. A path of length $m$ of $G$ is a sequence $x_{i_1},\ldots,x_{i_m}$  of vertices  with the property that
$(x_{i_k},x_{i_{k+1}})$ is a directed edge  of $G$ from $x_{i_k}$ to $x_{i_{k+1}}$ for $k=1,\ldots,m-1$. The path ideal of $G$ of length $m$ is a monomial ideal
$$
P_{m}=(\{x_{i_1}\cdots x_{i_m}\mid x_{i_1},\ldots,x_{i_m} \text{\ is a path of length\ } m \text{\ in\ } G\}).
$$
In particular, if $G$ is a cycle, then
$$P_m=(x_{i+1}\cdots x_{i+m}\mid i=1,2,\ldots,n),$$
where $x_{i+j}=x_{k}$ if $i+j\equiv k$ mod $n$ for $i\in [n]$, $j\in [m]$.
The following lemma holds from \cite[Corollary 5.1]{Zh}
and \cite[Theorem 3.3]{Zh1}.

 \begin{Lemma}\label{path} Let $G$ be a path or a cycle  with $n$ vertices and $P_{m}$ be its  path ideal  of length $m$.
Then $\reg\,(P_2)=\left\lfloor\frac{n+1}{3}\right\rfloor+1$ and $\reg\,(P_{n-1})=n-1$.
\end{Lemma}

 \begin{Lemma}\label{three1} Let $S=k[x_1,\ldots,x_6]$ be the  polynomial ring over a field $k$ and let $I=(x_1x_2,x_1x_3,x_2x_3,x_4x_5,x_4x_6,x_5x_6,x_1x_4,x_2x_5,x_3x_6)\subset S$ be a monomial ideal. Then we obtain  by  using CoCoA \cite{Co} that $\reg\,(I)=3$.
\end{Lemma}

\begin{Lemma}\label{depth5} Let $n\geq 3$ be an integer  and let
$$
I=\left(\prod\limits_{k=1}^{n}x_{1k},\prod\limits_{k=1}^{n}x_{2k}\right)+\sum\limits_{\begin{subarray}{c}
                             i,j\in [n] \\
                             i<j
                           \end{subarray}}\left(\prod\limits_{k=1}^{2}\frac{x_{k1}\cdots x_{k n}}{x_{ki}x_{kj}}\right)$$
 be a monomial ideal of  $S=k[x_{11},\ldots,x_{1n},x_{21},\ldots,x_{2n}]$. Then $\reg\,(I)=2n-3.$
\end{Lemma}
\begin{proof}
We may write  $I$ as $I=J+K$, where $K=(\prod\limits_{k=1}^{n}x_{1k},\prod\limits_{k=1}^{n}x_{2k},\prod\limits_{k=3}^{n}x_{1k}x_{2k})$ and $J$ is  a monomial ideal with $\mathcal {G}(J)=\mathcal {G}(I)\setminus \mathcal {G}(K)$. We sort all the elements of
$\mathcal {G}(J)$ in lexicographic order
  with $x_{11}>x_{12}>\cdots>x_{1n}>x_{21}>x_{22}>\cdots>x_{2n}$, that is:
$$\prod\limits_{k=1}^{2}\frac{x_{k1}\cdots x_{k n}}{x_{k(n-1)}x_{kn}}>
\prod\limits_{k=1}^{2}\frac{x_{k1}\cdots x_{k n}}{x_{k(n-2)}x_{kn}}>\cdots>
\prod\limits_{k=1}^{2}\frac{x_{k1}\cdots x_{k n}}{x_{k1}x_{k4}}>
\prod\limits_{k=1}^{2}\frac{x_{k1}\cdots x_{k n}}{x_{k1}x_{k3}}.$$
Let $u_s$ be the $s$-th element in $\mathcal {G}(J)$ for $s\in [r]$, where $r=|\mathcal{G}(J)|$.
Put $J_0=I$, $J_s=K+(u_{s+1},\ldots,u_r)$ for $s\in [r-1]$ and $J_r=K$.

Note that the degree of $u_s$ equals to $2n-4$ for any $s\in[r]$, by Remark \ref{add4}, we obtain that
$$\reg(I)=\reg(J_0)\leq \max\{\reg(J_{s}:u_{s})+2n-5, \reg(J_r)\mid \ s\in[r]\}$$
and the equality holds if $\reg(J_{1}:u_{1})+2n-4>\max\{\reg\,(J_{s}:u_{s})+2n-5,\  \reg\,(J_{r})\mid 2\leq s\leq r\}$.

\medskip
First, we  prove that $\reg\,(J_s:u_s)=2$ for any $s\in [r]$.

Since $u_s=\prod\limits_{k=1}^{2}\frac{x_{k1}\cdots x_{k n}}{x_{kp}x_{kq}}$ with $p,q\in [n]$, $p<q$ and $(p,q)\neq (1,2)$, one has
$$J_s=\left(\prod\limits_{k=1}^{n}x_{1k},\prod\limits_{k=1}^{n}x_{2k}\right)+\sum\limits_{\begin{subarray}{c}
                             1\leq i<p,\,i<j\leq n\\
                             \text{or\ }i=p,\, p<j<q
                           \end{subarray}}\left(\prod\limits_{k=1}^{2}\frac{x_{k1}\cdots x_{k n}}{x_{ki}x_{kj}}\right).$$
By a direct calculation, we obtain
\begin{eqnarray*}
J_s:u_s&=& \left(\prod\limits_{k=1}^{n}x_{1k},\prod\limits_{k=1}^{n}x_{2k}\right):u_s+\sum\limits_{\begin{subarray}{c}
                             1\leq i<p,\,i<j\leq n\\
                             \text{or\ }i=p,\, p<j<q
                           \end{subarray}}\left(\prod\limits_{k=1}^{2}\frac{x_{k1}\cdots x_{k n}}{x_{ki}x_{kj}}\right):u_s\\
&=& \left\{\begin{array}{ll}
(x_{11}x_{1q},x_{21}x_{2q},x_{1q}x_{2q})\ \ &\text{if} \ \ p=1,\\
(x_{1p}x_{1q},x_{2p}x_{2q},x_{1p}x_{2p},x_{1q}x_{2q})\ \ &\text{if} \ \ p\geq 2.\\
\end{array}\right.
\end{eqnarray*}
It can be regarded as the edge ideal of a path or cycle  with $4$ vertices, it follows from Lemma \ref{path} (1) that
$\reg\,(J_s:u_s)=2$.

Next, we   prove that $\reg\,(J_r)=2n-3$.

Let $J_{r}=J_{r+1}+(\prod\limits_{k=1}^{n}x_{2k})$ and
$J_{r+1}=(\prod\limits_{k=1}^{n}x_{1k},\prod\limits_{k=3}^{n}x_{1k}x_{2k})$.
By Lemma \ref{sqm}, one has
$$\reg\,(J_{r+1})=\reg\,((\prod\limits_{k=1}^{n}\!x_{1k},\prod\limits_{k=3}^{n}\!x_{1k}x_{2k}))
=|\text{supp}\,(J_{r+1})|-|\mathcal {G}(J_{r+1})|+1=2n-3.$$
Moreover, $\reg\,(J_{r+1}:\prod\limits_{k=1}^{n}x_{2k})=n-2$ because of $J_{r+1}:\prod\limits_{k=1}^{n}x_{2k}=(\prod\limits_{k=3}^{n}x_{1k})$.
Thus, by Corollary \ref{add1}, we have
$$\reg\,(J_r)=\max\{\reg\,(J_{r+1}:\prod\limits_{k=1}^{n}x_{2k})+n-1, \reg\,(J_{r+1}) \}=2n-3.$$
This concludes the proof.\end{proof}

\begin{Lemma}\label{depth6}
Let $n\geq 4$ be an integer  and let
$$I=\sum\limits_{j=1}^n\left(\frac{x_{11}\cdots x_{1n}}{x_{1j}}, \frac{x_{21}\cdots x_{2n}}{x_{2j}}\right)+
\sum\limits_{\begin{subarray}{c}
                             i,j\in [n-1], \\
                             i<j
                           \end{subarray}}\left(\prod\limits_{k=1}^{2}\frac{x_{k1}\cdots x_{k(n-1)}}{x_{kj}x_{kj}}\right)$$
 be a monomial ideal of $S=k[x_{11},\ldots,x_{1n},x_{21},\ldots,x_{2n}]$. Then
$\reg\,(I)=2n-5.$
\end{Lemma}
\begin{proof}
Let $I=J+K$, where $J=\sum\limits_{j=1}^{n-1}\left(\frac{x_{11}\cdots x_{1n}}{x_{1j}}, \frac{x_{21}\cdots x_{2n}}{x_{2j}}\right)$ and
$K$ is  a monomial ideal with $\mathcal {G}(K)=\mathcal {G}(I)\setminus \mathcal {G}(J)$. We sort all the elements of
$\mathcal {G}(J)$ in lexicographic order
  with $x_{11}>x_{12}>\cdots>x_{1n}>x_{21}>x_{22}>\cdots>x_{2n}$, that is:
$$\frac{x_{11}\cdots x_{1n}}{x_{1(n-1)}} >\cdots>
\frac{x_{11}\cdots x_{1n}}{x_{11}}>\frac{x_{21}\cdots x_{2n}}{x_{2(n-1)}} >\cdots>\frac{x_{21}\cdots x_{2n}}{x_{21}}.
$$
Let $u_s$ be the $s$-th element in $\mathcal {G}(J)$ and
 $J_{s}$ be  a monomial ideal with $\mathcal{G}(J_s)=\mathcal{G}(I)\setminus \{u_1,\ldots,u_s\}$  for $1\leq s\leq 2n-2$. Then $J_s=K+(u_{s+1},\ldots,u_{2n-2})$ for any $s\in [2n-3]$ and
$J_{2n-2}=K=\left(\prod\limits_{k=1}^{n-1}x_{1k},\prod\limits_{k=1}^{n-1}x_{2k}\right)+\sum\limits_{\begin{subarray}{c}
                             i,j\in [n-1], \\
                             i<j
                           \end{subarray}}\left(\prod\limits_{k=1}^{2}\frac{x_{k1}\cdots x_{k(n-1)}}{x_{kj}x_{kj}}\right)$.
It follows   from Lemma \ref{depth5} that $\reg\,(J_{2n-2})=2n-5$. Note that the degree of $u_s$ is $n-1$
for any $s\in[2n-2]$, by Remark \ref{add4}, we obtain that
$$\reg(I)\leq \max\{\reg(J_{s}:u_{s})+n-2, \reg(J_{2n-2})\mid \ s\in[2n-2]\}$$
and the equality holds if  $\reg(J_{1}:u_{1})+n-1>\max\{\reg\,(J_{s}:u_{s})+n-2,\  \reg\,(J_{2n-2})\mid 2\leq j\leq 2n-2\}$.

Now, we compute $\reg\,(J_s:u_s)$  for any $s\in [2n-2]$. We divide into the following  two cases:

(i) If $u_s=\frac{x_{11}\cdots x_{1n}}{x_{1p}}$ for some $p\in [n-1]$, then
\begin{eqnarray*}
K:u_s\!\!\!&=&\!\!\!\left(\prod\limits_{k=1}^{n-1}x_{1k},\prod\limits_{k=1}^{n-1}x_{2k}\right):u_s+\sum\limits_{\begin{subarray}{c}
                             i,j\in [n-1], \\
                             i<j
                           \end{subarray}}\left(\prod_{k=1}^{2}\frac{x_{k1}\cdots x_{k(n-1)}}{x_{kj}x_{kj}}\right):u_s\\
\!\!\!&=&\!\!\!(x_{1p},\prod\limits_{k=1}^{n-1}x_{2k})+\sum\limits_{j\in [n-1]\setminus \{p\}} \left(\prod_{k=1}^{2}\frac{x_{k1}\cdots x_{k(n-1)}}{x_{kp}x_{kj}}\right):\frac{x_{11}\cdots x_{1n}}{x_{1p}}\\
\!\!\!&=&\!\!\!(x_{1p},\prod\limits_{k=1}^{n-1}x_{2k})+\!\!\sum\limits_{j\in [n-1]\setminus \{p\}} \!\!\!\left(\frac{x_{21}\cdots x_{2(n-1)}}{x_{2p}x_{2j}}\right)\\
\end{eqnarray*}

\begin{eqnarray*}
	\!&=&\!(x_{1p})+\!\!\sum\limits_{j\in [n-1]\setminus \{p\}}\!\! \!\left(\frac{x_{21}\cdots x_{2(n-1)}}{x_{2p}x_{2j}}\right)
\end{eqnarray*}
where the second equality holds because $x_{1p}$ is a factor of  $\prod\limits_{k=1}^{2}\frac{x_{k1}\cdots x_{k(n-1)}}{x_{kj}x_{kj}}$ when $i,j\neq p$.
Therefore, one has
\[
J_s:u_s= K:u_s+(u_{s+1},\ldots,u_{2n-2}):u_s=(x_{1p})+\sum\limits_{j\in [n-1]\setminus \{p\}} \left(\frac{x_{21}\cdots x_{2(n-1)}}{x_{2p}x_{2j}}\right).
\]
Since $\sum\limits_{j\in [n-1]\setminus \{p\}} \left(\frac{x_{21}\cdots x_{2(n-1)}}{x_{2p}x_{2j}}\right)$ can be regarded as the path ideal of length $(n-3)$ of a cycle on the set $\{x_{21},\ldots,x_{2(n-1)}\}\setminus \{x_{2p}\}$. Thus,  by Lemmas \ref{sum1} and \ref{path}, we obtain
$$\reg\,(J_s:u_s)=\reg\,(\sum\limits_{j\in [n-1]\setminus \{p\}} \left(\frac{x_{21}\cdots x_{2(n-1)}}{x_{2p}x_{2j}}\right))=n-3.$$

(ii) If $u_s=\frac{x_{21}\cdots x_{2n}}{x_{2p}}$ for some $p\in [n-1]$, similar  to  (i), we obtain
$J_s:u_s=(x_{2p})+\sum\limits_{j\in [n-1]\setminus \{p\}} \left(\frac{x_{11}\cdots x_{1(n-1)}}{x_{1p}x_{1j}}\right)$
and $\reg\,(J_s:u_s)=n-3$. This finishes the proof.
 \end{proof}

\begin{Theorem}\label{m=3}
Let $\Delta_{3,n}$ be a chessboard complex with $n\geq 3$,  then
$$
\depth\,(S/\F(\Delta_{3,n}))=4.
$$
\end{Theorem}
 \begin{proof}
Let $\F^{\,\vee}$ be the Alexander dual of the facet ideal $\F(\Delta_{3,n})$. Then,  by Lemma \ref{ha} and Auslander-Buchsbaum formula, we have
 $$\depth\,(S/\F(\Delta_{3,n}))=3n-\text{pd}\,(S/\F(\Delta_{3,n}))=3n-\reg\,(\F^{\,\vee}).$$
We will prove that $\reg\,(\F^{\,\vee})=3n-4$,  the desired result follows.

By Theorem  \ref{decomposition}, one has
\begin{eqnarray*}
\F^{\,\vee}&=&\left(\prod\limits_{k=1}^{n}x_{1k},\prod\limits_{k=1}^{n}x_{2k},\prod\limits_{k=1}^{n}x_{3k}\right)+
\sum\limits_{\begin{subarray}{c}
                             i,j\in [n], \\
                             i<j
                           \end{subarray}}\left(\prod\limits_{k=1}^{3}\frac{x_{k1}\cdots x_{kn}}{x_{ki}x_{kj}}\right)+\\
&&\sum\limits_{j=1}^n\left(\prod\limits_{k=1,2}\frac{x_{k1}\cdots x_{kn}}{x_{kj}},
\prod\limits_{k=1,3}\frac{x_{k1}\cdots x_{kn}}{x_{kj}}, \prod\limits_{k=2,3}\frac{x_{k1}\cdots x_{kn}}{x_{kj}}\right).
\end{eqnarray*}
Let $J$ be a monomial ideal with $\mathcal {G}(J)=\mathcal {G}(\F^{\,\vee})\setminus \{\prod\limits_{k=1}^{n}x_{3k}\}$.
Claim: $\reg\,(J:\prod\limits_{k=1}^{n}x_{3k})=2n-3$ and $\reg(J)\leq 3n-4$. Thus the asserted results hold by Corollary \ref{add1}.

First, we  compute $\reg\,(J:\prod\limits_{k=1}^{n}x_{3k})$.
 Note that
$$J:\prod\limits_{k=1}^{n}x_{3k}=\sum\limits_{j=1}^n\left(\frac{x_{11}\cdots x_{1n}}{x_{1j}}, \frac{x_{21}\cdots x_{2n}}{x_{2j}}
\right)+
\sum\limits_{\begin{subarray}{c}
                             i,j\in [n], \\
                             i<j
                           \end{subarray}}\left(\prod\limits_{k=1}^{2}\frac{x_{k1}\cdots x_{kn}}{x_{ki}x_{kj}}\right).$$

If $n=3$, then $J:\prod\limits_{k=1}^{3}x_{3k}=\sum\limits_{k=1,2}(x_{k1}x_{k2},x_{k1}x_{k3},x_{k2}x_{k3})+
(x_{11}x_{21},x_{12}x_{22},x_{13}x_{23})$.
Thus $\reg\,(J:\prod\limits_{k=1}^{3}x_{3k})=3=2n-3$ by Lemma \ref{three1}.

If $n\geq 4$.  Set
$W=\left\{\prod\limits_{k=1}^{2}\frac{x_{k1}\cdots x_{k(n-1)}}{x_{kp}x_{kq}}\mid  p,q\in [n-1], p<q\right\}$,
we sort all the elements of $W$ in lexicographic order  with $x_{11}>x_{12}>\cdots>x_{1n}>x_{21}>x_{22}>\cdots>x_{2n}$, that is:
$$\prod\limits_{k=1}^{2}\frac{x_{k1}\cdots x_{k(n-1)}}{x_{k(n-2)}x_{k(n-1)}}>
\prod\limits_{k=1}^{2}\frac{x_{k1}\cdots x_{k(n-1)}}{x_{k(n-3)}x_{k(n-1)}}>\cdots>
\prod\limits_{k=1}^{2}\frac{x_{k1}\cdots x_{k(n-1)}}{x_{k1}x_{k3}}>
\prod\limits_{k=1}^{2}\frac{x_{k1}\cdots x_{k(n-1)}}{x_{k1}x_{k2}}.$$
Let $u_s$ be the $s$-th element in $W$ and $I_s=(J:\prod\limits_{k=1}^{n}x_{3k})+(u_1,\ldots,u_s)$  for  $s\in [r]$. Then
$$I_r=(J:\prod\limits_{k=1}^{n}x_{3k})+(W)=\sum\limits_{j=1}^n\left(\frac{x_{11}\cdots x_{1n}}{x_{1j}}, \frac{x_{21}\cdots x_{2n}}{x_{2j}}
\right)+
\sum\limits_{\begin{subarray}{c}
                             i,j\in [n-1], \\
                             i<j
                           \end{subarray}}\left(\prod\limits_{k=1}^{2}\frac{x_{k1}\cdots x_{k(n-1)}}{x_{ki}x_{kj}}\right).$$
It follows from Lemma \ref{depth6} that $\reg\,(I_r)=2n-5. \hspace{5.0cm}(1)$

Now, we compute $\reg\,(I_{s-1}:u_s)$ for any $s\in [r]$.

Set $\prod\limits_{k=1}^{2}\frac{x_{k1}\cdots x_{k(n-1)}}{x_{kp}x_{kq}}$ for some  $p,q\in[n-1]$ and $p<q$. By simple calculations, one has
\begin{eqnarray*}
(J:\prod\limits_{k=1}^{n}x_{3k}):u_s\!\!&=&\!\!
\sum\limits_{j=1}^n\left(\frac{x_{11}\cdots x_{1n}}{x_{1j}}, \frac{x_{21}\cdots x_{2n}}{x_{2j}}
\right):u_s+
\sum\limits_{\begin{subarray}{c}
                             i,j\in [n], \\
                             i<j
                           \end{subarray}}\left(\prod\limits_{k=1}^{2}\frac{x_{k1}\cdots x_{kn}}{x_{ki}x_{kj}}\right):u_s\\
&=&\sum\limits_{k=1,2}(x_{kp}x_{kq},x_{kp}x_{kn},x_{kq}x_{kn})+(x_{1p}x_{2p},x_{1q}x_{2q},x_{1n}x_{2n}).
\end{eqnarray*}
Hence, it is easy to check that
\begin{eqnarray*}
I_{s-1}:u_s&=&((J:\prod\limits_{k=1}^{n}x_{3k})+(u_{i}\mid i\in [s-1])):u_s\\
&=& (x_{1p}x_{1q},x_{1p}x_{1n},x_{1q}x_{1n},x_{2p}x_{2q},x_{2p}x_{2n},x_{2q}x_{2n},x_{1p}x_{2p},x_{1q}x_{2q},x_{1n}x_{2n}),
\end{eqnarray*}
where $(u_{i}\mid i\in [0])=(0)$. It follows from Lemma \ref{three1} that
$$\reg\,(I_{s-1}:u_s)=3.\eqno(2)$$
Therefore, by Remark \ref{add2} and formulas (1), (2), one has $\reg\,(J:\prod\limits_{k=1}^{3}x_{3k})=2n-3$.

Finally, we   prove $\reg\,(J)\leq 3n-4$.

We may write $J$ as $J=\left(\prod\limits_{k=1}^{n}x_{1k},\prod\limits_{k=1}^{n}x_{2k}\right)+K_1+K_2$ where $K_2=\sum\limits_{\begin{subarray}{c}
                             i,j\in [n], \\
                             i<j
                           \end{subarray}}\left(\prod\limits_{k=1}^{3}\frac{x_{k1}\cdots x_{kn}}{x_{ki}x_{kj}}\right)$ and $K_1=\sum\limits_{j=1}^n\left(\prod\limits_{k=1,2}\!\!\frac{x_{k1}\cdots x_{kn}}{x_{kj}},
\prod\limits_{k=1,3}\!\!\frac{x_{k1}\cdots x_{kn}}{x_{kj}}, \prod\limits_{k=2,3}\!\!\frac{x_{k1}\cdots x_{kn}}{x_{kj}}\right)$.
We sort all the elements of $\mathcal {G}(K_1)$ and $\mathcal {G}(K_2)$ in lexicographic order
  with $x_{11}>x_{12}>\cdots>x_{1n}>x_{21}>x_{22}>\cdots>x_{2n}>x_{31}>x_{32}>\cdots>x_{3n}$ respectively.
Let $w_{p}$ and $w_{r_1+q}$  be the $p$-th element of $\mathcal {G}(K_1)$ and the $q$-th element of $\mathcal {G}(K_2)$ respectively, where $p\in [r_1]$, $r_1=|\mathcal {G}(K_1)|$, $q\in [r_2]$ and $r_2=|\mathcal {G}(K_2)|$.
Put $J_s=\left(\prod\limits_{k=1}^{n}x_{1k},\prod\limits_{k=1}^{n}x_{2k}\right)+(w_{s+1},\ldots,w_r)$  for any $s\in [r-1]$, where   $r=r_1+r_2$ and $J_r=\left(\prod\limits_{k=1}^{n}x_{1k},\prod\limits_{k=1}^{n}x_{2k}\right)$. It follows  that $\reg\,(J_r)=2n-1\leq 3n-4$ from Lemma  \ref{sum1} (3).

Claim:  $\reg\,(J_s:w_s)+d_s\leq 3n-3$  for any $s\in[r]$, where $d_s$ is the degree of  $w_s$. Thus we obtain that $\reg\,(J)\leq 3n-4$ by  Remark \ref{add4}.

We prove this claim by considering the following two cases:

(i) If $w_{s}\in \mathcal {G}(K_1)$,
then we can assume that $w_s=\prod\limits_{k=\ell_1,\ell_2}\frac{x_{k1}\cdots x_{kn}}{x_{kp}}$ for $1\leq \ell_1< \ell_2\leq 3$ and   $p\in [n]$. Thus
\begin{eqnarray*}
J_s:w_s&=&\left((\prod\limits_{k=1}^{n}x_{1k},\prod\limits_{k=1}^{n}x_{2k})+(w_{s+1},\ldots,w_{r})\right):w_s\\
&=&\left(\prod\limits_{k=1}^{n}x_{1k}, \prod\limits_{k=1}^{n}x_{2k}\right)\!:\!w_s\!+ \! (w_{s+1},\ldots,w_{r_1})\!:\!w_s
\!+\!\sum\limits_{\begin{subarray}{c}
                             i,j\in [n], \\
                             i<j
                           \end{subarray}}\left(\prod\limits_{k=1}^{3}\!\frac{x_{k1}\cdots x_{kn}}{x_{ki}x_{kj}}\right)\!:\!u_s\\
&=& \left\{\begin{array}{ll}
(x_{1p},x_{2p})+J'_s\ \ &\text{if}\  \ell_1=1\text{\ and\ } \ell_2=2,\\
(x_{\ell_1p})+J'_s\  &\text{otherwise.}
\end{array}\right.
\end{eqnarray*}
where $\{\ell_3\}=[3]\setminus \{\ell_1,\ell_2\}$, $J'_s=\sum\limits_{j\in [n]\setminus \{p\}} \left(\frac{x_{\ell_{31}}\cdots x_{\ell_{3n}}}{x_{\ell_{3p}}x_{\ell_{3j}}}\right)$.

 Note that $J'_s$ may be regarded as the path ideal of length $(n-2)$ of a cycle  on the vertex set $\{x_{\ell_31},\ldots,x_{\ell_3n}\}\setminus\{x_{\ell_3p}\}$. By Lemmas \ref{sum1} and \ref{path} (2), one has
$$\reg(J_s:w_s)+d_s=\reg(J'_s)+d_s=(n-2)+2(n-1)=3n-4<3n-3.\eqno(3)$$

(ii) If $w_{s}\in \mathcal {G}(K_2)$, then
$w_s=\prod\limits_{k=1}^{3}\frac{x_{k1}\cdots x_{kn}}{x_{kp}x_{kq}}$ for some $p,q\in [n]$, where $p<q$. By direct calculations, we obtain

\begin{eqnarray*}
J_s:w_s\!\!&=&\!\!\left\{\begin{array}{ll}
\left(\prod\limits_{k=1}^{n}x_{1k}, \prod\limits_{k=1}^{n}x_{2k}\right):w_s\ \ &\text{if}\ \ s=r\\
\left(\prod\limits_{k=1}^{n}x_{1k}, \prod\limits_{k=1}^{n}x_{2k}\right):w_s+(w_{s+1},\ldots,w_{r}):w_s\ \ &\text{if} \ \ s<r\\
\end{array}\right.\\
	\!\!&=&\!\!\left\{\begin{array}{ll}
(x_{11}x_{12},x_{21}x_{22})\  &\text{if} \ s=r\\
(x_{1p}x_{1q},x_{2p}x_{2q})\!+\!\sum\limits_{\begin{subarray}{c}
                             i=p,\, j<q\\
                             \text{or}\ 1\leq i< p,j \leq n
                           \end{subarray}}\!\!\!\left(\prod\limits_{k=1}^{3}\frac{x_{k1}\cdots x_{kn}}{x_{ki}x_{kj}}\right):w_s\ &\text{if}  \ s<r\\
\end{array}\right.\\
\!\!&=&\!\! \left\{\begin{array}{ll}
(x_{11}x_{12},x_{21}x_{22})\ \ &\text{if}\ \ p=1\text{\ and\ } q=2\\
(x_{1p}x_{1q},x_{2p}x_{2q},x_{1q}x_{2q}x_{3q})\ \ &\text{if} \ \ p=1\text{\ and\ } q\geq 3\\
(x_{1p}x_{1q},x_{2p}x_{2q},x_{1p}x_{2p}x_{3p},x_{1q}x_{2q}x_{3q})\ \ &\text{otherwise}.
\end{array}\right.
\end{eqnarray*}
Hence   $\reg\,(J_s:w_s)=3$ if $p=1$ from  the expression of $J_s:w_s$  and Lemma \ref{sqm}; Otherwise,
let $J''_s=(x_{1p}x_{1q},x_{2p}x_{2q},x_{1p}x_{2p}x_{3p})$, then $J''_s:x_{1q}x_{2q}x_{3q}=(x_{1p},x_{2p})$.
From Lemma \ref{sqm}, $\reg\,(J''_s)=3$ and $\reg\,(J''_s:x_{1q}x_{2q}x_{3q})=1$. Hence
$\reg\,(J_s:w_s)=3$ by Corollary \ref{add1}.
This implies  $\reg\,(J_s:w_s)+d_s=3+3(n-2)=3n-3$.
This proof is completed.
\end{proof}

 Let $I\subset S$  be a monomial ideal. The $a$-invariant of $S/I$, denoted by $a(S/I)$, is the degree (as a rational function) of
the Hilbert series of $S/I$.

\begin{Corollary}
Let $\Delta_{m,n}$ be a chessboard complex with $n\geq m$. If $m\leq 3$, then
$$
a(S/\F(\Delta_{m,n}))=0.
$$
 \end{Corollary}
 \begin{proof} By \cite[Theorem 3.20]{MV}, we have $0\leq a(S/\F(\Delta_{m,n}))\leq \reg\,(S/\F(\Delta_{m,n}))-\depth\,(S/\F(\Delta_{m,n}))=0$, as desired.
\end{proof}

We conclude this paper with the following natural conjectures.
\begin{Conjecture}\label{chessbord}
Let $\Delta_{m,n}$ be a chessboard complex with $n\geq m\geq 1$, then
$$
\reg\,(S/\F(\Delta_{m,n}))=\depth\,(S/\F(\Delta_{m,n}))=2(m-1).
$$
 \end{Conjecture}

Of course, the above conjecture is true if $m\le 3$ by Theorem \ref{facetcover} and Theorem \ref{m=3}. By using CoCoA \cite{Co}, we get that $\reg\,(S/\F(\Delta_{4,4}))=\depth\,(S/\F(\Delta_{4,4}))=6$,
which  is consistent with the above conjecture.

\medskip
\hspace{-6mm} {\bf Acknowledgments}

 \vspace{3mm}
\hspace{-6mm}
The authors are grateful to the computer algebra system CoCoA \cite{Co} for providing us with a large number of examples.
This research is supported by the Natural Science Foundation of Jiangsu Province (No. BK20221353)  and  by foundation of the Priority Academic Program Development of Jiangsu Higher Education Institutions.
Finally, the authors would like to thank  referees who read carefully the manuscript and gave very helpful comments, which improved the paper both in mathematics and presentation.

\end{document}